\documentclass[11pt,a4paper]{amsart}
\usepackage[utf8]{inputenc}
\usepackage{amsmath}
\usepackage{amsfonts}
\usepackage{url}
\usepackage{fullpage}
\usepackage{amsthm}
\usepackage{xcolor}
\usepackage{amssymb}
\usepackage{cleveref}

\newtheorem{theorem}{Theorem}[section]
\newtheorem{proposition}[theorem]{Proposition}
\newtheorem{lemma}[theorem]{Lemma}
\newtheorem{corollary}[theorem]{Corollary}        
 
\theoremstyle{definition}
\newtheorem{remark}[theorem]{Remark} 

\DeclareMathOperator{\Inn}{Inn}
\DeclareMathOperator{\Aut}{Aut}
\DeclareMathOperator{\Sym}{Sym}
\DeclareMathOperator{\GL}{GL}
\DeclareMathOperator{\AGL}{AGL}
\DeclareMathOperator{\soc}{soc}
\DeclareMathOperator{\pndr}{\nu}

\title{Derangements in permutation groups  with two orbits} 
\author{Melissa Lee}
\address{School of Mathematics, Monash University, Clayton VIC 3800, Australia}
\email{melissa.lee@monash.edu}
\author{Tomasz Popiel}
\address{School of Mathematics, Monash University, Clayton VIC 3800, Australia}
\email{tomasz.popiel@monash.edu}
\author{Gabriel Verret}
\address{Department of Mathematics, University of Auckland, Private Bag 92019, Auckland 1142, New Zealand}
\email{g.verret@auckland.ac.nz}
\begin{document}

\begin{abstract}
A classical theorem of Jordan asserts that if a group $G$ acts transitively on a finite set of size at least $2$, then $G$ contains a derangement (a fixed-point free element). Generalisations of Jordan's theorem have been studied extensively, due in part to their applications in graph theory, number theory and topology. We address a generalisation conjectured recently by Ellis and Harper, which says that if $G$ has exactly two orbits and those orbits have equal length $n \geq 2$, then $G$ contains a derangement. We prove this conjecture in the case where $n$ is a product of two primes, and verify it computationally for $n \leq 30$.
\end{abstract}

\maketitle

\section{Introduction}
We assume throughout  that all  groups and sets are finite. Let $G$ be a group acting on a set $\Omega$. A {\em derangement} is an element of $G$ that fixes no point of $\Omega$. A classical theorem of Jordan \cite{Jordan1872} asserts that if $G$ acts transitively on $\Omega$ and $|\Omega|\geq 2$, then $G$ contains a derangement. Equivalently, a group is never equal to the union of the conjugates of a proper subgroup. Generalisations of this result have been and continue to be studied intensively, due in part to numerous well-known applications \cite{serre}. This is intended to be a short note on one such generalisation, so we refer the reader to e.g. \cite[Chapter~1]{BG} for further historical background. 

The transitivity assumption in Jordan's theorem is necessary: it is easy to construct examples with $G$ having exactly two orbits but no derangement. On the other hand, Ellis and Harper have conjectured  that this cannot happen if the two orbits have equal length $n \geq 2$; see \cite[Conjecture~2]{harper_ellis}. They have proved this conjecture in various cases, including 
\begin{itemize} 
    \item[(i)] when $G$ acts primitively on one of its two orbits;
    \item[(ii)] when $G$ is simple, nilpotent, or of order at most $1000$;
    \item[(iii)] when $n$ is a prime power.
\end{itemize}

The main aim of this note is to prove Ellis and Harper's conjecture under the assumption that $n$ is a product of two primes; see \Cref{mainCor}. We also verify computationally that the conjecture holds when $ n \leq 30$; see \Cref{computations}.

\section{Preliminaries}
\begin{lemma}[{\cite[Lemma 2.2]{friends}}]\label{friends}
Let $G$ be a group acting transitively on a set $\Omega$, let $p$ be a prime and let $P$ be a Sylow $p$-subgroup of $G$. The minimum length of a $P$-orbit on $\Omega$ is the largest power of $p$ dividing $|\Omega|$.
\end{lemma}

\begin{lemma}\label{DerangementsInPGroups1}
Let $V = \mathbb{F}_q^d$, where $d\geq 2$ and $q$ is a prime power. Fix a basis for $V$, let $0\neq a\in V$ and define $U=\{v\in V\mid a\cdot v=0\}$. If $k$ is the number of non-zero coefficients of $a$ with respect to the fixed basis, then the number of $(u_1,\dots,u_d) \in U$ such that $u_1,\dots,u_d \neq 0$ is  equal to
\[\frac{(q-1)^{d-k+1}}{q} \left( (q-1)^{k-1}-(-1)^{k-1} \right).
\] 
In particular, this number is at most $(q-1)^{d-1}$. 
\end{lemma}
\begin{proof}
Given $j\geq 1$, $m\in \mathbb{F}_q$ and $(a_1,\ldots,a_j)\in \mathbb{F}_q^j$ with $a_1,\dots,a_j \neq 0$, consider the equation $\sum_{i=1}^j a_ix_i=m$. Call a solution $(x_1, \dots, x_j) \in \mathbb{F}_q^j$ to this equation `good' if $x_1, \dots, x_j \neq 0$. First notice that, because $\mathbb{F}_q$ is a field, the number of good solutions does not depend on $a_1,\dots,a_j$. Similarly, the number of good solutions depends only on whether $m=0$ or not. We may thus define $D_m(j)$ to be the number of good solutions to this and hence every equation of the given form. For $j\geq 2$, we have 
\[
D_0(j) = (q-1)D_{-a_jx_j}(j-1) = (q-1)D_1(j-1).
\]
It follows that, for $j\geq 2$, we have
\[
D_1(j)  = (q-2)D_1(j-1)+ D_0(j-1). 
\]
Combining these two equations yields, for $j \geq 3$, 
\[
D_1(j) =  (q-2)D_1(j-1)+(q-1)D_1(j-2).
\]
This is a second-order linear difference equation for $D_1(j)$; the initial conditions  $D_1(1)=1$ and  $D_1(2)=q-2$ yield
\[
D_1(j) = \frac{(q-1)^{j}-(-1)^j}{q}, 
\qquad \text{and hence} \qquad
D_0(j) = (q-1)\frac{(q-1)^{j-1}-(-1)^{j-1}}{q}
\]
for all $j\geq 1$. The first assertion of the lemma follows upon observing that the number of  $(u_1,\dots,u_d)\in U$ such that  $u_1,\dots,u_d \neq 0$ is equal to $(q-1)^{d-k}D_0(k)$. Since $a\neq 0$, we have $k\geq 1$ and the second assertion is then easily verified. 
\end{proof}

\begin{lemma}\label{DerangementsInPGroups2}
Let $q$ be a prime power and let $V = \mathbb{F}_q^d$. 
If $\mathcal{W}$ is a set of subspaces of $V$ of codimension $1$ such that 
\begin{enumerate}
\item $\bigcup_{W\in \mathcal{W}} W = V$ and 
\item $\bigcap_{W\in \mathcal{W}} W = \{0\}$,
\end{enumerate}
then $2 \leq d \leq |\mathcal{W}|-q+1$.
\end{lemma}
\begin{proof}
The assertion that $d \geq 2$ is immediate, because Condition~1 is never satisfied if $d\leq 1$. Let $W_1\in \mathcal{W}$. If $W_1\neq \{0\}$, then, by Condition~2, there exists $W_2\in \mathcal{W}$ such that $W_1\nleq W_2$. It follows that $W_1+W_2=V$ and $\dim(W_1+W_2)=\dim(W_1)+\dim(W_2)-\dim(V)=2(d-1)-d=d-2$, so $W_1\cap W_2$ has codimension $2$ in $V$. If $W_1\cap W_2\neq \{0\}$, then, by Condition~2, there exists $W_3\in \mathcal{W}$ such that $W_1\cap W_2\nleq W_3$. Since $W_3$ has codimension $1$ in $V$, this implies that $W_3+(W_1\cap W_2)=V$ and, by a similar calculation as earlier, $W_1\cap W_2\cap W_3$ has codimension $3$ in $V$. Repeating this procedure, we find that $\mathcal{W}$ contains $d$ subspaces $W_1,\ldots,W_d$ with $\bigcap_{i\in \{1,\ldots,d\}} W_i = \{0\}$. By choosing an appropriate basis for $V$, we can assume that $W_i$ is defined by the linear equation $x_i=0$. Note that  $\bigcup_{i\in\{1,\ldots,d\} } W_i$ consists of the vectors in $V$ with at least one coordinate equal to $0$, so $|\bigcup_{i\in\{1,\ldots,d\} } W_i|=q^d-(q-1)^d$. This leaves $(q-1)^d$ elements of $V$ to `cover' by adjoining further subspaces from $\mathcal{W}$, in order to satisfy Condition 1. \Cref{DerangementsInPGroups1} implies that adjoining one further subspace from $\mathcal{W}$ covers at most $(q-1)^{d-1}$ further elements.  We must therefore adjoin at least another $q-1$ subspaces in order to satisfy Condition 2, so $|\mathcal{W}| \geq d + (q-1)$.
\end{proof}

\begin{remark}
Note that the upper bound in~\Cref{DerangementsInPGroups2} is tight. Indeed, for $i\in\{1,\ldots, d\}$, let $W_i$ be the subspace defined by $x_i=0$. Note that $\bigcap_{i\in\{1,\ldots d\}} W_i = \{0\}$. Let $\mathcal{U}$ be the set of subspaces that are strictly between $V$ and $W_1\cap W_2$. Since $W_1\cap W_2$ has codimension $2$ in $V$, we have $|\mathcal{U}|=q+1$. It is easy to check that $\bigcup_{U\in \mathcal{U}} U = V$ and that $\mathcal{U} \cap \{W_1,\ldots,W_d\}=\{W_1,W_2\}$, hence $\mathcal{U} \cup \{W_1,\ldots,W_d\}$ has size $d+q-1$ and satisfies both conditions of~\Cref{DerangementsInPGroups2}. Moreover, it is not hard to show using \Cref{DerangementsInPGroups1} that, up to conjugacy in $\GL(V)$, this is the unique tight example, but we will not need this fact.
\end{remark}

Given a group $G$ acting (not necessarily faithfully) on a set $\Omega$ and $\omega\in\Omega$, we write $G_\omega$ for the point stabiliser of $\omega$ in $\Omega$. Given $g\in G$ and a subset $\Delta\subseteq\Omega$ preserved by $G$, we write $G_\Delta$ for the setwise stabiliser of $\Delta$ in $G$ and $g^\Delta$ for the permutation induced by $g$ on $\Delta$. We also let $G^\Delta:=\{g^\Delta\colon g\in G\}$.

Let $n\geq 1$ and let $T$ be a non-abelian finite simple group. Recall that the full wreath product $\Aut(T)\wr S_n$ has socle $K := T_1 \times \dots \times T_n$, where each $T_i$ is isomorphic to $T$. Moreover, $\{T_1,\ldots,T_n\}$ is the set of minimal normal subgroups of $T_1\times\cdots\times T_n$, hence $\Aut(T)\wr S_n$ acts on $\{T_1,\ldots,T_n\}$ by conjugation.  Given $G\leq \Aut(T)\wr S_n$, the stabiliser $G_{T_1}$ of $T_1$ acts on $T_1$, and the induced permutation group $G_{T_1}^{T_1}$ is naturally identified with a subgroup of $ \Aut(T_1)$. 
In particular, we have $T_1^{T_1}=\Inn(T_1)$ under this identification.

\begin{proposition}\label{WreathProducts}
With notation as above, if
\begin{enumerate}
\item $\Inn(T_1)\leq G_{T_1}^{T_1}$ and 
\item $G$ acts primitively on $\{T_1,\ldots,T_n\}$,
\end{enumerate}
then either $G\cap K=1$, $G\cap K\cong T$ or $K\leq G$.
\end{proposition}
\begin{proof}
Let $\varphi_i$ be the natural projection from $K$ to $T_i$ and let 
$$K_i := \mathrm{ker}\varphi_i=T_1\times\cdots \times T_{i-1}\times 1 \times T_{i+1}\times\cdots\times T_n.$$ 
Note that $\Inn(T_1)$ is the unique minimal normal subgroup of $G_{T_1}^{T_1}$, because $\Inn(T_1)\leq G_{T_1}^{T_1}\leq \Aut(T_1)$. Since $G\cap K$ is a normal subgroup of $G_{T_1}$, it follows that $(G\cap K)^{T_1}$ is normal in $G_{T_1}^{T_1}$, and so either $(G\cap K)^{T_1}=1$ or $\Inn(T_1)\leq (G\cap K)^{T_1}$. If $(G\cap K)^{T_1}=1$, then $G\cap K\leq K_1$. Since $G$ acts transitively on $\{T_1,\ldots,T_n\}$, this implies that $G\cap K\leq K_i$ for every $i\in\{1,\ldots,n\}$, and thus $G\cap K=1$, as required. We may thus assume that $\Inn(T_1)\leq (G\cap K)^{T_1}$, i.e. the restriction $\varphi_1:(G\cap K)\to T_1$ is surjective and thus $(G\cap K)/(G\cap K_1)\cong T$. 

For $i\in\{1,\ldots,n\}$, let $G_i=G\cap K_i$. Define an equivalence relation $\sim$ on $\{T_1,\ldots,T_n\}$ by $T_i \sim T_j$ if and only if $G_i = G_j$.  This equivalence relation is $G$-invariant so, by Condition~(2), the induced partition is either the universal one or the partition into singletons. We now consider these two cases separately.

Case 1: $G_i = G_j$ for all $i,j\in \{1,\ldots,n\}$. Since $K_1\cap\cdots \cap K_n=1$, we have $G_1=1$, so the restriction $\varphi_1:G\cap K\to T_1$ is injective. We saw earlier that it is surjective, hence it is an isomorphism and $G\cap K\cong T$, as required.

Case 2: $G_i \neq G_j$ for all $i,j$ with $i\neq j$. We saw earlier that $(G\cap K)/G_1\cong T$. Together with Condition~2,  this implies that $(G\cap K)/G_i\cong T$ for all $i\in\{1,\ldots,n\}$. We now proceed by induction. Suppose that $1\leq m<n$ is such that $(G\cap K)/(G_1  \cap \dots \cap G_m)\cong T^m$. Let $N=(G_1  \cap \dots\cap G_m) G_{m+1}$. Note that  $(G_1  \cap \dots\cap G_m)$ and $G_{m+1}$ are both normal subgroups of $G\cap K$, hence so is $N$. Since $(G\cap K)/G_{m+1}\cong T$ is simple, we must have $N=G\cap K$ or $N = G_{m+1}$. If $N = G_{m+1}$, then $G_1 \cap \dots\cap G_m \leq G_{m+1}$. Since $(G\cap K) /(G_1  \cap \dots\cap G_m) \cong T^m$ has precisely $m$ normal subgroups of index $|T|$, it follows that $G \cap K$ has precisely $m$ normal subgroups of index $|T|$ containing $G_1 \cap \dots \cap G_m$. The latter are precisely $G_1,\dots,G_m$, so we must have $G_{m+1}=G_i$ for some $i \in \{1,\dots,m\}$, which is a contradiction. Therefore, $N = G \cap K$. It follows that
\[
(G\cap K)/(G_1 \cap \dots \cap G_{m+1})\cong (G\cap K)/(G_1 \cap \dots \cap G_m)\times (G\cap K)/G_{m+1} \cong T^m \times T \cong T^{m+1}.
\]
This completes the induction and yields $(G\cap K)/(G_1 \cap \dots \cap G_n) \cong T^n$. Since $G \cap K \leq K \cong T^n$, it follows that $G \cap K=K$ and hence $K \leq G$, as required.
\end{proof}

Given a group $G$ acting on a set $\Omega$, let 
$$\pndr(G)=\frac{|\bigcup_{\omega \in \Omega} G_\omega|}{|G|}.$$ 
Note that this is exactly the proportion of non-derangements in $G$. 

\begin{lemma} \label{lemma:proportion}
If $G$ is a group acting on a set $\Omega=\Omega_1\cup\cdots\cup\Omega_k$ where each of the subsets $\Omega_1,\dots,\Omega_k$ is preserved by $G$, then
$$\pndr(G)\leq \sum_{i\in\{1,\ldots,k\}} \pndr(G^{\Omega_i}).$$
\end{lemma}
\begin{proof}

For each $i \in \{1,\dots,k\}$, let $K_i$ be the kernel of the action homomorphism from $G$ to $\Sym(\Omega_i)$. We have
\begin{align*}
\pndr(G)&=\frac{|\bigcup_{\omega \in \Omega} G_\omega|}{|G|} \leq \sum_{i\in\{1,\ldots,k\}} \frac{|\bigcup_{\omega \in \Omega_i} G_\omega|}{|G|} = \sum_{i\in\{1,\ldots,k\}} \frac{|\bigcup_{\omega \in \Omega_i} G_\omega^{\Omega_i}|\cdot |K_i|}{|G^{\Omega_i}|\cdot |K_i|}\\
&=\sum_{i\in\{1,\ldots,k\}} \frac{|\bigcup_{\omega \in \Omega_i} G_\omega^{\Omega_i}|}{|G^{\Omega_i}|}=\sum_{i\in\{1,\ldots,k\}} \pndr(G^{\Omega_i}),
\end{align*}
as required.
\end{proof}

\section{Main results}
Throughout this section, let $G$ be a permutation group on a set $\Omega$ of size $2n$, and assume that $G$ has exactly two orbits, $\Omega_1$ and $\Omega_2$, each of length $n$. 

\begin{proposition}\label{computations}
If $2\leq n\leq 30$, then $G$ has a derangement.
\end{proposition}

\begin{proof}
Note that the transitive groups of degree at most $30$ are known \cite{MR2168238} and readily accessible in Magma~\cite{magma}. The proof is supported by the Magma code available on the first author's GitHub~\cite{mel_github}; the process that we now describe is carried out for each degree $n \in \{2,\dots,30\}$ using the function \texttt{HalfTransitiveDerangements}. If $G$ has no derangements, then \cite[Corollary~4]{harper_ellis} implies that there exist transitive but imprimitive permutation groups $G_1$ and $G_2$ of degree $n$ such that $G^{\Omega_i} \cong G_1$ for $i \in \{1,2\}$. We begin by constructing each transitive but imprimitive permutation group $H$ of degree $n$ and directly counting the number of non-derangements in $H$ (by computing the support of a representative of each conjugacy class of $H$). By \Cref{lemma:proportion}, the proportion of non-derangements in $G$ is bounded above by the sum of the proportions of the non-derangements in $G_1$ and $G_2$, which we compute for each pair $(G_1,G_2)$ using the aforementioned data. If this proportion is less than~$1$, then $G$ certainly has a derangement, so we discard the pair $(G_1,G_2)$. For each remaining pair $(G_1,G_2)$, we know that $G$ must be a subdirect product of $G_1$ and $G_2$. The property of being a subdirect product is preserved when taking supergroups, so we simply traverse the upper layers of the subgroup lattice of $G_1 \times G_2$, stopping whenever we find a group that is not a subdirect product, until we have constructed all subdirect products of $G_1$ and $G_2$ (up to conjugacy in $G_1\times G_2$). Finally, we exhibit a derangement in each subdirect product either by random search or by checking each conjugacy class. 
\end{proof}

\begin{lemma}\label{SameSizeOrbits}
Let $p$ be a prime, let $P$ be a Sylow $p$-subgroup of $G$ and let $n=bp^k$. If $b<p$, then $P$ has $2b$ orbits of length $p^k$.
\end{lemma}
\begin{proof}
By~\Cref{friends}, the minimum length of a $P^{\Omega_i}$-orbit is $p^k$. Since $b<p$, we have $n<p^{k+1}$, hence every $P^{\Omega_i}$-orbit has length exactly $p^k$ and the result follows.
\end{proof}

\begin{lemma}\label{NewBounds}
Let $p$ be a prime, let $P$ be a Sylow $p$-subgroup of $G$, let $|P|=p^d$ and let $n=bp$ with $b<p$. Then $P$ has $2b$ orbits of length $p$, $P$ is elementary abelian and $|\{P_\omega\mid\omega\in\Omega\}|\leq 2b$. Moreover, if $P$ does not contain a derangement, then $2\leq d\leq |\{P_\omega\mid\omega\in\Omega\}|-p+1\leq 2b-p+1$. In particular, if $b<(p+1)/2$, then $P$ contains a derangement.
\end{lemma}
\begin{proof}
Note that $\bigcap_{\omega\in\Omega} P_\omega =1$ because $G$ is a permutation group on $\Omega$. By~\Cref{SameSizeOrbits}, $P$ has $2b$ orbits of length $p$. It follows that for every $\omega\in\Omega$, we have $|P:P_\omega|=p$, so $P_\omega$ is a normal (and maximal) subgroup of $P$. This implies that $P$ is elementary abelian (because $P$ has trivial Frattini subgroup) and that $|\{P_\omega\mid\omega\in\Omega\}|\leq 2b$. If $P$ does not contain a derangement, then $\bigcup_{\omega\in\Omega} P_\omega =P$ and~\Cref{DerangementsInPGroups2} implies that $2\leq d\leq |\{P_\omega\mid\omega\in\Omega\}|-p+1$. 
\end{proof}

\begin{proposition}\label{LastDay}
If $n=pq$ with $p>q$ primes and $q$ does not divide $p-1$, then $G$ has a derangement.
\end{proposition}
\begin{proof}
Assume for a contradiction that $G$ does not have a derangement. Let $P$ be a Sylow $p$-subgroup of $G$ and write $|P|=p^d$. By~\Cref{NewBounds}, $P$ has $2q$ orbits of length $p$, $P$ is elementary abelian, $|\{P_\omega\mid\omega\in\Omega\}|\leq 2q$ and  $2\leq d\leq |\{P_\omega\mid\omega\in\Omega\}|-p+1\leq 2q-p+1$.

Suppose first that $|P^{\Omega_i}|\leq p$ for some $i\in\{1,2\}$. This implies that $P_\omega=P_{\omega'}$ for every $\omega,\omega'\in\Omega_i$, namely $|\{P_\omega\mid\omega\in\Omega_i\}|=1$, hence $|\{P_\omega\mid\omega\in\Omega\}|\leq q+1$, contradicting the facts that $2\leq d \leq |\{P_\omega\mid\omega\in\Omega\}|-p+1$ and $q<p$. Therefore, $|P^{\Omega_i}|\geq p^2$ for both $i\in\{1,2\}$.

By \cite[Corollary~4]{harper_ellis}, each $G^{\Omega_i}$ is imprimitive, and so preserves a block system consisting of either $p$ blocks of size $q$ or $q$ blocks of size $p$. In the former case, $G^{\Omega_i} \leq S_q \wr S_p$, so $|P^{\Omega_i}|\leq p$, contradicting $|P^{\Omega_i}|\geq p^2$. (Here we abuse notation and write $\leq$ to mean ``is isomorphic to a subgroup of''.) Therefore, each $G^{\Omega_i}$ preserves a system $\mathcal{B}_i$ of $q$ blocks of size $p$, so $G^{\Omega_i} \leq S_p \wr S_q$. Let $B\in\mathcal{B}_i$. Then $G^{\Omega_i}\leq (G^{\Omega_i})_B^B\wr (G^{\Omega_i})^{\mathcal{B}_i}$, with $(G^{\Omega_i})_B^B$ a transitive subgroup of $\text{Sym}(B) \cong S_p$ and $(G^{\Omega_i})^{\mathcal{B}_i}$ a transitive subgroup of $\text{Sym}(\mathcal{B}_i) \cong S_q$. Since $p$ is prime, it follows from classical results of Burnside \cite[Chapter~IX, Theorem~IX]{Burnside} that either $(G^{\Omega_i})_B^B \leq \AGL(1,p)$ or $(G^{\Omega_i})_B^B$ is almost simple. 

Suppose first that $(G^{\Omega_i})_B^B$ is almost simple, and note that the socle $T$ of $(G^{\Omega_i})_B^B$ is transitive and, in particular, $p$ divides $|T|$. Let $K=\soc ((G^{\Omega_i})_B^B\wr S_q)\cong T^q$. By~\Cref{WreathProducts}, either $G^{\Omega_i}\cap K=1$, $G^{\Omega_i}\cap K\cong T$ or $K\leq G^{\Omega_i}$. If $G^{\Omega_i}\cap K=1$ or $G^{\Omega_i}\cap K\cong T$, then  $|P^{\Omega_i}|\leq p$, contradicting $|P^{\Omega_i}|\geq p^2$. If $K\leq G^{\Omega_i}$, then $p^q$ divides $|G^{\Omega_i}|$ and $q\leq d$, contradicting $d \leq 2q-p+1$, $q<p$ and $q\neq p-1$.

We must therefore have $(G^{\Omega_i})_B^B \leq \AGL(1,p)$, hence $G^{\Omega_i}\leq \AGL(1,p)\wr S_q$, for both $i\in\{1,2\}$. Let $Q$ be a Sylow $q$-subgroup of $G$. Since $q$ does not divide $p-1$, it does not divide $|\AGL(1,p)|=p(p-1)$, hence $|Q^{\Omega_i}|=q$. It follows from \Cref{friends} that all orbits of $Q^{\Omega_i}$  have length $q$ and thus $|Q:Q_\omega|=q$ for every $\omega\in\Omega$. Let $\omega_i\in\Omega_i$ and $g\in Q_{\omega_i}$. Note that $g$ stabilises the block of $\mathcal{B}_i$ containing $\omega_i$, and hence, since $g$ is a $q$-element, $g$ must also stabilise the remaining $q-1$ blocks of $\mathcal{B}_i$. It follows that $g^{\Omega_i}\in\AGL(1,p)^q$, but $q$ does not divide $|\AGL(1,p)|$, hence  $g^{\Omega_i}=1$. This implies that $g\in Q_\omega$ for every $\omega\in\Omega_i$, so $Q_{\omega_i} \leq Q_\omega$ for every $\omega \in \Omega_i$. Since $|Q_\omega|=|Q_{\omega_i}|$, we have $Q_\omega=Q_{\omega_i}$ for every $\omega\in\Omega_i$. This implies that $Q$ has at most two distinct point stabilisers, so it must contain a derangement (because it cannot equal the union of two proper subgroups).
\end{proof}

\begin{corollary} \label{mainCor}
If $n=pq$ with $p,q$ primes, then $G$ has a derangement.
\end{corollary}
\begin{proof}
If $p=q$, then the result follows from \Cref{NewBounds}, so we assume that $p>q$. If $q$ does not divide $p-1$, then the result follows from~\Cref{LastDay}, so we  assume  that $q$ divides $p-1$. If $q=p-1$, then $p=3$ and $q=2$ and the result follows from~\Cref{computations}. Otherwise, $q\leq (p-1)/2$ and the result follows from~\Cref{NewBounds}.
\end{proof}

\section*{Acknowledgements}
Melissa Lee acknowledges the support of an Australian Research Council Discovery Early Career Researcher Award (project number DE230100579). Melissa Lee and Tomasz Popiel are grateful to Gabriel Verret and The University of Auckland for their hospitality during a research visit in January 2025, when this work was initiated.

\bibliographystyle{plain}
\bibliography{refs}

\end{document}